\documentclass[10pt,twoside,a4paper]{article}
\usepackage{amsfonts,amssymb,amsmath,amscd,latexsym,makeidx,theorem}
\author{Luca Martinazzi\thanks{The first author was supported by the ETH Research Grant no. ETH-02 08-2.}\\ \small{ETH Zurich} \\ \footnotesize{\texttt{luca@math.ethz.ch}}
\and Mircea Petrache\\ \small{ETH Zurich} \\ \footnotesize{\texttt{mircea.petrache@math.ethz.ch}}}
\title{Asymptotics and quantization for a mean-field equation of higher order}

\newtheorem{trm}{Theorem}

\newtheorem{cor}[trm]{Corollary}
\newtheorem{lemma}[trm]{Lemma}

\newcommand{\R}[1]{\mathbb{R}^{#1}}

\newcommand{\de}{\partial}
\newcommand{\ve}{\varepsilon}
\newcommand{\bs}{\backslash}

\newcommand{\norm}[1]{\left\Arrowvert{#1}\right\Arrowvert}
\newcommand{\abs}[1]{\left\arrowvert{#1}\right\arrowvert}

\newenvironment{proof}{\noindent\emph{Proof.}}{\hfill$\square$\medskip}

\newenvironment{rmk}{\medskip\noindent\emph{Remark.}}{\hfill$\bullet$\medskip}

\DeclareMathOperator{\diver}{div}
\DeclareMathOperator{\loc}{loc}

\DeclareMathOperator*{\dist}{dist}

\begin{document}
\maketitle

\begin{abstract}
\noindent Given a regular bounded domain $\Omega\subset\R{2m}$, we describe the limiting behavior of sequences of solutions to the mean field equation of order $2m$, $m\geq 1$,
$$(-\Delta)^m u=\rho \frac{e^{2mu}}{\int_\Omega e^{2mu}dx}\quad\text{in }\Omega,$$
under the Dirichlet boundary condition and the bound $0<\rho\leq C$. We emphasize the connection with the problem of prescribing the $Q$-curvature.
\end{abstract}

\section{Introduction}
Let $\Omega\subset\R{2m}$ be a bounded domain with smooth boundary. Given a sequence of numbers $\rho_k>0$,
we consider solutions to the mean-field equation of higher order
\begin{equation}\label{eq2}
(-\Delta)^m u_k=\rho_k\frac{e^{2mu_k}}{\int_{\Omega}e^{2mu_k}dx}
\end{equation}
subject to the Dirichlet boundary condition
\begin{equation}\label{dir}
u_k=\partial_\nu u_k=\ldots=\de_\nu^{m-1}u_k=0\quad \textrm{on }\partial \Omega.
\end{equation}
As shown in Corollary 8 of \cite{mar1}, every $u_k$ is smooth. In this paper we study the limiting behavior of the sequence $(u_k)$. We show that concentration-compactness phenomena together with geometric quantization occur. We particularly emphasize the interesting relationship with the thriving problem of prescribing the $Q$-curvature.

For any $\xi\in\overline \Omega$, let $G_\xi(x)$ denote the Green function of the operator $(-\Delta)^m$ on $\Omega$ with Dirichlet boundary condition (see e.g. \cite{ACL}), i.e
\begin{equation}\label{greenf}
\left\{
\begin{array}{ll}
(-\Delta)^m G_\xi=\delta_\xi&\text{in }\Omega\\
G_\xi=\de_\nu G_\xi=\ldots=\de_\nu^{m-1} G_{\xi}=0 & \text{on }\de\Omega.
\end{array}
\right.
\end{equation}
Also fix any $\alpha\in[0,1)$. We then have
\begin{trm}\label{trm2} Let $u_k$ be a sequence of solutions to \eqref{eq2}, \eqref{dir} and assume that
$$0<\rho_k\leq C.$$ Then one of the following is true:
\begin{itemize}
\item[(i)] Up to a subsequence $u_k\to u_0$ in $C^{2m-1,\alpha}(\overline\Omega)$ for some $u_0\in C^\infty(\overline\Omega)$.
\item[(ii)] Up to a subsequence, $\lim_{k\to\infty}\max_\Omega u_k=\infty$ and there is a positive integer $N$ such that
\begin{equation}\label{quant}
\lim_{k\to\infty}\rho_k=N\Lambda_1,\quad \Lambda_1=(2m-1)!|S^{2m}|.
\end{equation}
Moreover there exists a non-empty finite set $S=\{x^{(1)},\ldots, x^{(N)}\}\subset\Omega$
such that 
\begin{equation}\label{green}
u_k\to \Lambda_1\sum_{i=1}^N G_{x^{(i)}}\quad \text{in }C^{2m-1,\alpha}_{\loc}(\overline\Omega\backslash S).
\end{equation}
\end{itemize}
\end{trm}
The mean field equation in dimensions $2$ and $4$ has been object of intensive study in the recent years. We refer e.g. to \cite{NS}, \cite{wei}, \cite{RW} and the references therein. In particular in \cite{RW} the $4$-dimensional analogous of our Theorem \ref{trm2} was proved, and many of the ideas developed there are used in our treatment. 

\medskip

The geometric constant $\Lambda_1$ showing up in \eqref{quant} and \eqref{green} is the total $Q$-curvature\footnote{For the definition of $Q$-curvature we refer to \cite{cha}, or to the introduction of \cite{mar1} and the references therein.} of the round $2m$-dimensional sphere. It is worth explaining how this relation with Riemannian geometry arises. It will be shown in Lemma \ref{claim4} below that one can blow up the $u_k$'s at suitably chosen \emph{concentration points}, and get in the limit a solution $u_0$ to the Liouville equation
\begin{equation}\label{liou}
(-\Delta)^m u_0=(2m-1)!e^{2m u_0}\quad\text{in }\R{2m}
\end{equation}
with the bound
\begin{equation}\label{area}
\int_{\R{2m}}e^{2mu_0}dx<\infty.
\end{equation}
Geometrically, if $u_0$ solves \eqref{liou}-\eqref{area}, then the conformal metric $e^{2u_0}g_{\R{2m}}$ on $\R{2m}$ (where $g_{\R{2m}}$ is the Euclidean metric) has constant $Q$-curvature equal to $(2m-1)!$ and finite volume. As shown in \cite{CC}, there are many such conformal metrics on $\R{2m}$, and the crucial step in Lemma \ref{claim4} below is to show that
\begin{equation}\label{eta0}
u_0(x)=:\log\bigg(\frac{2}{1+|x|^2}\bigg).
\end{equation}
The above function has the property that $e^{u_0}g_{\R{2m}}=(\pi^{-1})^*g_{S^{2m}}$, where $g_{S^{2m}}$ is the round metric on $S^{2m}$, and $\pi:S^{2m}\to\R{2m}$ is the stereographic projection. In particular
\begin{equation}\label{vol}
\int_{\R{2m}}e^{2mu_0}dx=|S^{2m}|.
\end{equation}
This is the basic reason why the constant $\Lambda_1$ appears in Theorem \ref{trm2}. 
In order to show that \eqref{eta0} holds, we use the classification result of \cite{mar1} and a technique of \cite{RS}, which allows us to rule out all the solutions of \eqref{liou} which are ``non-spherical'', hence whose total $Q$-curvature might be different from $\Lambda_1$.

\medskip

We will further exploit such connections with conformal geometry mainly by referring to Theorem 1 in \cite{mar2}, about the concentration-compactness phenomena for sequences of conformal metrics on $\R{2m}$ with prescribed $Q$-curvature (compare \cite{BM}, \cite{ARS} and \cite{Rob} for $2$ and $4$-dimensional analogous results). We state a simplified version of this theorem in the appendix, since we shall use it several times.

\medskip

The last crucial step in the proof of Theorem \ref{trm2} is the generalization to arbitrary dimension of a clever argument of Robert-Wei \cite{RW} based on a Pohozaev-type identity, which rules out blow-up points at the boundary (see Lemma \ref{claim11}) and allows to sharply estimate the energy concentrating at each blow-up point (see Lemma \ref{claim13})

\medskip

One can also state Theorem \ref{trm2} as an eigenvalue problem, as in \cite{wei}. In this case one replaces the term $\frac{\rho_k}{\int_\Omega e^{2mu_k}}$ by the constant $\lambda_k>0$ in \eqref{eq2}, so we consider the equation
\begin{equation}\label{eq1}
(-\Delta)^m u_k=\lambda_k e^{2mu_k}.
\end{equation}
The assumption $0<\rho_k\leq C$ gets replaced by
\begin{equation}\label{sigmak}
\Sigma_k:=\int_{\Omega}\lambda_k e^{2mu_k}dx\leq C,
\end{equation}
and we keep the boundary condition \eqref{dir}. Then Theorem \ref{trm2} implies that either
\begin{itemize}
\item[(i)] up to a subsequence $u_k\to u_0$ in $C^{2m-1,\alpha}_{\loc}(\overline \Omega)$, or
\item[(ii)] up to a subsequence $\Sigma_k\to N\Lambda_1$ and $(u_k)$ satisfies \eqref{green}, with the same notation of Theorem \ref{trm2}.
\end{itemize}

\medskip

Several times we use standard elliptic estimates. For the interior estimates one can safely rely on \cite{GT} or \cite{GM}. For the estimates up to the boundary, one can refer to \cite{ADN}. Throughout the paper the letter $C$ denotes a large universal constant which does not depend on $k$ and can change from line to line, or even within the same line.

\section{Proof of Theorem \ref{trm2}}

The proof will be organized as follows. We shall see in Corollary \ref{coru}, that if $\sup_{\Omega} u_k\leq C$, then $u_k$ is bounded in $C^{2m-1,\alpha}(\overline\Omega)$ and case (i) of Theorem \ref{trm2} occurs. Therefore, after Corollary \ref{coru} we shall assume that
\begin{equation}\label{ukinfty}
\lim_{k\to\infty}\sup_{\Omega} u_k=\infty,
\end{equation}
and prove that case (ii) of Theorem \ref{trm2} occurs. Let
\begin{equation}\label{alphak}
\alpha_k:=\frac{1}{2m}\log\bigg(\frac{(2m-1)!\int_\Omega e^{2mu_k}dx}{\rho_k}\bigg),\quad \hat u_k:=u_k-\alpha_k.
\end{equation}

\begin{lemma}\label{claim1} 
Up to selecting a subsequence, we have $\alpha_k\geq -C$.
\end{lemma}

\begin{proof} Indeed
\begin{equation}\label{eqhat}
(-\Delta)^m\hat u_k= (2m-1)!e^{2m \hat u_k}\quad \textrm{in }\Omega
\end{equation}
and
$$\hat u_k=-\alpha_k,\quad \partial_\nu \hat u_k=\ldots=\partial_\nu^{m-1}\hat u_k=0\quad \textrm{on }\partial\Omega.$$
Moreover
\begin{equation}\label{eq3.2}
\int_{\Omega}e^{2m\hat u_k}dx=\frac{\rho_k}{(2m-1)!}\leq C.
\end{equation}
Using Green's representation formula, we infer
\begin{equation}\label{eq3.3}
\hat u_k(x)=(2m-1)!\int_\Omega G_x(y)e^{2m \hat u_k(y)}dy-\alpha_k.
\end{equation}
Then, integrating \eqref{eq3.3} and using \eqref{eq3.2}, the fact that $\|G_y\|_{L^1(\Omega)}\leq C$, with $C$ independent of $y$, and the symmetry of $G$, i.e. $G_x(y)=G_y(x)$, we get
\begin{equation}\label{eq3.4}
\int_\Omega |\hat u_k+\alpha_k|dx\leq C.
\end{equation}
Now, according to Theorem \ref{main} in the Appendix, we have that one of the following is true:
\begin{itemize}
\item[(i)] $\hat u_k\to \hat u_0$ in $C^{2m-1,\alpha}_{\loc}(\Omega)$ for some function $\hat u_0$.
\item[(ii)] $\hat u_k\to-\infty$ locally uniformly in $\Omega\backslash \Omega_0$, for some closed nowhere dense (possibly empty) set $\Omega_0$ of Hausdorff dimension at most $2m-1$.
\end{itemize}
In both cases the claim of the lemma easily follows from \eqref{eq3.4}.
\end{proof}

\begin{cor}\label{coru} The following facts are equivalent:
\begin{itemize}
\item[(i)] Up to selecting subsequences, $u_k\leq C$.
\item[(ii)] Up to selecting subsequences, $\hat u_k\leq C$.
\item[(iii)] Up to selecting subsequences, $u_k\to u_0$ in $C^{2m-1,\alpha}(\overline\Omega)$ for some smooth function $u_0$.
\end{itemize}
\end{cor}

\begin{proof} (i) $\Rightarrow $ (ii) follows at once from Lemma \ref{claim1}.

\medskip

\noindent (ii) $\Rightarrow$ (iii) follows by elliptic estimates, observing that
$$|(-\Delta)^m u_k|=|(-\Delta)^m \hat u_k|=\big|(2m-1)!e^{2m \hat u_k}\big|\leq C$$
and using \eqref{dir}.

\medskip

\noindent (iii)$\Rightarrow$ (i) is obvious.
\end{proof}

\begin{lemma}\label{claim2} For all $\ell\in\{1,\ldots,2m-1\}$ and for $p\in[1,\tfrac{2m}{\ell})$, there exists $C=C(\ell,p)$ such that
\begin{equation}\label{nablai}
\int_{B_R(x_0)}|\nabla^\ell \hat u_k|^p dx\leq C R^{2m-ip},
\end{equation}
for any $B_R(x_0)\subset\Omega$.
\end{lemma}

\begin{proof}
We prove the claim by duality. Let $\varphi\in C^{\infty}_c(\Omega)$ and $q=\tfrac{p}{p-1}$. Differentiating \eqref{eq3.3}, using Fubini's theorem, the relation $G_x(y)=G_y(x)$ and the estimate (see \cite{DAS})
\begin{equation}\label{DAS}
|\nabla^\ell G_y(x)|\leq \frac{C}{|x-y|^\ell},
\end{equation}
we get
\begin{eqnarray*}
 \int_{B_R(x_0)}|\nabla^\ell\hat u_k|\varphi dx&\leq&C\int_{B_R(x_0)}\left(\int_\Omega\abs{\nabla^\ell G_y(x)}e^{2m \hat u_k
(y)}dy\right)\abs{\varphi(x)}dx\\
&\leq& C\int_\Omega e^{2m \hat u_k(y)}\bigg(\int_{B_R(x_0)}\abs{x-y}^{-\ell}\abs{\varphi(x)}dx\bigg)dy\\
&\leq& C\norm{\varphi}_{L^q(\Omega)}\int_\Omega e^{2m \hat u_k(y)}\bigg(\int_{B_R(x_0)}\frac{dx}{|x-y|^{\ell p}}\bigg)^\frac{1}{p} dy\\
&\leq& C\norm{\varphi}_{L^q(\Omega)} R^{\frac{2m}{p}-\ell},
\end{eqnarray*}
where in the last inequality we used $p<\tfrac{2m}{\ell}$, \eqref{eq3.2}, and the simple estimate
$$\int_{B_R(x_0)}\frac{dx}{|x-y|^{\ell p}}\leq \int_{B_R(y)}\frac{dx}{|x-y|^{\ell p}}\leq CR^{2m-\ell p}.$$
The lemma follows at once.
\end{proof}

\begin{lemma}\label{claim3}
Let $x_k\in\Omega$ be such that
\begin{equation}\label{ukin}
u_k(x_k)=\max_\Omega u_k\to\infty.
\end{equation}
Let $\mu_k:=2e^{-\hat u_k(x_k)}$. Then $\frac{\dist(x_k,\de\Omega)}{\mu_k}\rightarrow+\infty$.
\end{lemma}

\begin{proof}
Suppose that the conclusion of the lemma is false. Then the rescaled sets
$$\Omega_k:=\tfrac{1}{\mu_k}(\Omega - x_k)$$
converge, up to rotation, to $(-\infty,t_0)\times\R{2m-1}$ for some $t_0\geq 0$. Define
\begin{equation}\label{utilde}
\tilde u_k(x):=\hat u_k(x_k + \mu_k x) +\log(\mu_k),\quad x\in\Omega_k.
\end{equation}
By \eqref{ukin} and Corollary \ref{coru} we have $\mu_k\rightarrow 0$.  Fix $R>0$ such that $B_R(0)\cap\de\Omega_k\neq\emptyset$, and let $x\in B_R(0)\cap\Omega_k$. Then, for $1\leq \ell\leq 2m-1$, using \eqref{eq3.3} and \eqref{DAS}, we get
\begin{eqnarray*}
 \abs{\nabla^\ell\tilde u_k(x)}
&\leq & C   \mu_k^\ell\int_\Omega|\nabla^\ell G_{x_k + \mu_kx}(y)|e^{2m\hat u_k(y)}dy\\
&\leq& C\mu_k^\ell\bigg(\int_{\Omega\setminus B_{2R\mu_k}(x_k)}\frac{1}{\abs{x_k + \mu_k x - y}^\ell}e^{2m\hat u_k(y)}dy\\
&&+\int_{B_{2R\mu_k}(x_k)}\frac{1}{|x_k + \mu_k x - y|^\ell}e^{2m\hat u_k(y)}dy\bigg)\\
&\leq& CR^{-\ell}\int_\Omega e^{2m\hat u_k}dy+ C\mu_k^{\ell-2m}\int_{B_{2R\mu_k}(x_k)}\frac{dy}{\abs{x_k + \mu_k x - y}^\ell}\\
&\leq& C(R),
\end{eqnarray*}
where we used that for $y\in \Omega\setminus B_{2R\mu_k}(x_k)$ and $x\in B_R(0)\cap \Omega_k$ we have $R\mu_k\leq \abs{x_k + \mu_k x - y}$ and, for any $y\in \Omega$ we have $e^{2m\hat u_k(y)}\leq 2^{2m}\mu_k^{-2m}$. This implies
$$|\tilde u_k(x)-\tilde u_k(0)|\leq C(R)|x|\quad \textrm{for }|x|\leq R.$$
Choosing $x\in B_R(0)\cap \partial \Omega_k$ we get $|u_k(x_k)|=\abs{\hat u_k(x_k)+\alpha_k}\leq C(R)$, contradicting \eqref{ukin}.
\end{proof}

\begin{rmk}\label{claim3bis}
In the choice of the scales $\mu_k$ we are free to some extent. Our particular choice is made in order to give a cleaner form to the blow-up limit described in Lemma \ref{claim4} and to make the connection with the problem of prescribing the $Q$-curvature more transparent.
\end{rmk}

From now on we shall assume that \eqref{ukinfty} holds.

\begin{lemma}\label{claim4}
Let $\tilde u_k$ be defined as in \eqref{utilde}. Then, up to selecting a subsequence, we have
\begin{equation}\label{eqclaim4} \lim_{k\rightarrow +\infty}\tilde u_k(x)=\log\bigg(\frac{2}{1+\abs{x}^2}\bigg)\quad \textrm{in }C^{2m-1,\alpha}_{\loc}(\R{2m}).
\end{equation}
\end{lemma}

\begin{proof} We give the proof in two steps.

\medskip

\noindent\emph{Step 1.}
We first claim that up to a subsequence, $\tilde u_k\to \tilde u_0$ in $C^{2m-1,\alpha}_{\loc}(\R{2m})$, for some smooth function $\tilde u_0$ satisfying
\begin{equation}\label{u0}
(-\Delta)^m\tilde u_0=(2m-1)!e^{2m\tilde u_0}.
\end{equation}
Let us first assume $m>1$. We apply Theorem \ref{main} on $\R{2m}$ to the sequence $(\tilde u_k)$, where it is understood that one has to invade $\R{2m}$ with bounded sets and extract a diagonal subsequence in order to get the local convergence on all of $\R{2m}$. Since $\tilde u_k\leq \log 2$, we have $S_1=\emptyset$, in the notation of Theorem \ref{main}. Then one of the following is true:

\begin{itemize}
\item[(i)] $\tilde u_k\to \tilde u_0$ in $C^{2m-1,\alpha}_{\loc}(\R{2m})$ for some function $\tilde u_0\in C^{2m-1,\alpha}_{\loc}(\R{2m})$, or
\item[(ii-a)] $\tilde u_k\to -\infty$ locally uniformly in $\R{2m}$ (case $S_0=\emptyset$), or
\item[(ii-b)] there exists a closed nowhere dense set $S_0\neq \emptyset$ of Hausdorff dimension at most $2m-1$ and numbers $\beta_k\to\infty$ such that
$$\frac{\tilde u_k}{\beta_k}\to\varphi\textrm{ in }C^{2m-1,\alpha}_{\loc}(\R{2m}\backslash S_0),$$
where
\begin{equation}\label{eq71}
\Delta^m\varphi\equiv 0,\quad \varphi\leq 0,\quad \varphi\not\equiv 0 \textrm{ on }\R{2m},\quad \varphi\equiv 0\textrm{ on }S_0.
\end{equation}
\end{itemize}
Since $\tilde u_k(0)=\log 2$, (ii-a) can be ruled out. Assume now that (ii-b) occurs. From Liouville's theorem and \eqref{eq71}, we get $\Delta\varphi\not\equiv 0$, hence for some $R>0$ we have
$\int_{B_R(0)}|\Delta\varphi|dx>0$ and
\begin{equation}\label{72}
\lim_{k\to \infty}\int_{B_R}|\Delta \tilde u_k|dx=\lim_{k\to\infty}\beta_k\int_{B_R(0)}|\Delta\varphi|dx=+\infty.
\end{equation}
By \eqref{nablai}, and using the change of variables $y=x_k+\mu_k x$, we get, for $1\leq j\leq m-1$,  
\begin{eqnarray}
\int_{B_R(0)}|\Delta^j \tilde u_k|dx&=& \mu_k^{-2m+2j}\int_{B_{R\mu_k}(x_k)} |\Delta^j \hat u_k|dy
  \nonumber\\
&\leq&C \mu_k^{-2m+2j}(R\mu_k)^{2m-2j}\leq C R^{2m-2j},\label{deltaj}
\end{eqnarray}
which contradicts \eqref{72} for $j=1$ and any fixed $R>0$. Hence (i) occurs. Clearly $\tilde u_0$ satisfies \eqref{u0} and our claim is proved.

For the case $m=1$, we infer from Theorem 3 in \cite{BM} that either case (i) or (ii-a) above occur, and case (ii-a) is ruled out as above. 

\medskip

\noindent\emph{Step 2.} We now want to prove that $\tilde u_0=\log\frac{2}{1+|x|^2}$. From Fatou's lemma and \eqref{eq3.2} we infer
\begin{eqnarray*}
\int_{\R{2m}}e^{2m\tilde u_0}dx&=&\lim_{R\to\infty}\int_{B_R(0)}e^{2m\tilde u_0}dx \leq \lim_{R\to\infty}\liminf_{k\to\infty}\int_{B_R(0)}e^{2m\tilde u_k}dx\\
&=&\lim_{R\to\infty}\liminf_{k\to\infty}\int_{B_{R\mu_k}(x_k)}e^{2m\hat u_k}dx \leq\int_{\Omega}e^{2m\hat u_k}dx\leq C.
\end{eqnarray*}
If $m=1$, then our claim follows directly from \cite{CL}. Assume now $m>1$.  From Theorem 2 in \cite{mar1} we get that either
\begin{equation}\label{eqstsol}
\tilde u_0=\log\frac{2\lambda}{1+\lambda^2\abs{x-x_0}^2}
\end{equation}
for some $\lambda>0$ and $x_0\in\R{2m}$, or there exists $j\in\{1,\ldots,m-1\}$ such that 
\begin{equation}\label{deltau}
 \Delta^j\tilde u_0(x)\rightarrow a\text{ as }\abs x\rightarrow +\infty, 
\end{equation}
for some constant $a<0$. On the other hand, \eqref{deltau}
implies that for every $R>0$ large enough there is $k(R)\in\mathbb{N}$ such that
$$\int_{B_R(0)}|\Delta^j \tilde u_k|dx\geq \frac{|a|}{2}|B_R(0)|\geq \frac{R^{2m}}{C},\quad \text{for }k\geq k(R).$$
This contradicts \eqref{deltaj} in the limit as $R\to 0$, whence \eqref{eqstsol} has to hold.
Since $\tilde u_k(0)=\max_{\Omega_k}\tilde u_k=\log 2$, the same facts hold for $\tilde u_0$. Therefore $x_0=0$ and $\lambda=1$ in \eqref{eqstsol}. This proves our second claim, hence the lemma.
\end{proof}

\begin{lemma}\label{hp} There are $N>0$ converging sequences $x_{k,i}\to x^{(i)}$, $1\leq i\leq N$,  with $\lim_{k\to\infty}u_k(x_{k,i})=\infty$ such that, setting
\begin{equation}\label{ki2}
\tilde u_{k,i}(x):=\hat u_k(x_{k,i}+\mu_{k,i}x)+\log\mu_{k,i},\quad 
\mu_{k,i}:=2e^{-\hat u_k(x_{k,i})},\\
\end{equation}
we have
\begin{enumerate}
 \item[($A_1$)] $\lim_{k\to\infty}\frac{\abs{x_{k,i} - x_{k,j}}}{\mu_{k,i}} +\infty$ for $1\leq i\neq j\leq N$,
 \item[($A_2$)] $\lim_{k\to \infty}\frac{\dist(x_{k,i},\de\Omega)}{\mu_{k,i}}=+\infty$, for $1\leq i\leq N$
 \item[($A_3$)] $\tilde u_{k,i}\rightarrow \eta_0$ in $C^{2m-1,\alpha}_{\loc}(\R{2m})$, for $1\leq i\leq N$, where $\eta_0(x)=\log\Big(\frac{2}{1+\abs{x}^2}\Big)$.
\item[($A_4$)] For $1\leq i\leq N$
\begin{equation}\label{energiai}
\lim_{R\to\infty}\lim_{k\to\infty}\int_{B_{R\mu_{k,i}}(x_{k,i})}e^{2m\hat u_k}dx=|S^{2m}|.
\end{equation}
\item[($A_5$)] $\inf_{1\leq i\leq N}|x-x^{(i)}|^{2m} e^{2m\hat u_k(x)}\leq C$ for every $x\in \Omega$.
\end{enumerate}
\end{lemma}

\begin{proof} We proceed inductively. 

\medskip

\noindent\emph{Step 1.}
For $N=1$, choose $x_{k,1}$ such that $u_k(x_{k,1})=\sup_\Omega u_k$. Then Lemma \ref{claim3} and Lemma \ref{claim4} imply that $(x_{k,1})$ satisfies $(A_2)$ and $(A_3)$. Moreover $(A_1)$ is empty and $(A_4)$ follows at once from $(A_3)$ \eqref{vol}. If also $(A_5)$ is satisfied, we are done. Otherwise we construct a new sequence, as in the inductive step below.

\medskip

\noindent\emph{Step 2.} Assume that $\ell$ sequences $\{(x_{k,i})\to x^{(i)}:1\leq i\leq \ell$\}, have been constructed so that they satisfy $(A_1)$, $(A_2)$, $(A_3)$ and $(A_4)$, but not $(A_5)$.
Set
$$w_k(x):=\inf_{1\leq i\leq \ell}|x-x_{k,i}|^{2m}e^{2m\hat u_k(x)},$$
so that $\lim_{k\to\infty}\sup_\Omega w_k=\infty$, and choose $y_k\in\Omega$ such that $w_k(y_k)=\sup_{\Omega}w_k$.
Then $y_k\to y$ up to a subsequence. Also set
\begin{equation}\label{eqclaim5}
\gamma_k=2e^{-\hat u_k(y_k)},\qquad v_k(x)=\hat u_k(y_k+\gamma_k x)+\log \gamma_k.
\end{equation}
We claim that $(A_1)$, $(A_2)$, $(A_3)$ and $(A_4)$ hold for the $\ell+1$ sequences
$$\{(x_{k,i})\to x^{(i)}:1\leq i\leq \ell+1\},$$
if we set
$$
\left\{
\begin{array}{l}
x_{k,\ell+1}:=y_k\\
x^{(\ell+1)}:=y\\
\tilde u_{k,\ell+1}:=v_k\\
\mu_{k,\ell+1}:=\gamma_k
\end{array}
\right.
$$
Since $w_k(y_k)\rightarrow +\infty$ we get
$$\lim_{k\to\infty}\frac{|y_k-x_{k,i}|}{\gamma_k}\geq \lim_{k\to\infty}\frac{w_k(y_k)^\frac{1}{2m}}{2}=+\infty\quad\text{for } 1\leq i\leq \ell.$$
We claim that we also have
$$\lim_{k\to\infty}\frac{\abs{y_k-x_{k,i}}}{\mu_{k,i}}=+\infty\quad \text{for }1\leq i\leq \ell.$$ Indeed, setting $\theta_{k,i}:=\frac{y_k-x_{k,i}}{\mu_{k,i}}$, we have
$$
 \abs{y_k-x_{k,i}}^{2m}e^{2m\hat u_k(y_k)}=\abs{\theta_{k,i}}^{2m}\exp(2m[\hat u_k(x_{k,i} +\mu_{k,i}\theta_{k,i})+\log\mu_{k,i}]).
$$
If our claim were false, then the right-hand side would be bounded thanks to $(A_3)$, but then we would have
$w_k(y_k)\leq C,$
against our assumption. This proves $(A_1)$.
Fix now $\ve, R>0$. Since $\max w_k$ is attained at $y_k$, and using \eqref{eqclaim5}, we have 
\begin{equation}
 e^{2mv_k(x)}\leq 2^{2m}\frac{\inf_{1\leq i\leq \ell}|y_k-x_{k,i}|^{2m}}{\inf_{1\leq i\leq \ell}|y_k+\gamma_k x-x_{k,i}|^{2m}}.
\end{equation}
Choose $k(\ve,R)$ such that $\abs{y_k - x_{k,i}}\geq\frac{R}{\ve}\gamma_k$ for $k\geq k(\ve,R)$ and $1\leq i\leq \ell$. Then
$$\frac{\abs{y_k-x_{k,i}}}{\abs{y_k-x_{k,i} +\gamma_k x}}\leq \frac{1}{1-\ve}\quad  \text{for } x\in B_R(x), \;k\geq k(\ve,R),\;1\leq i\leq \ell,$$
hence
$$e^{2mv_k(x)}\leq \frac{2^{2m}}{(1-\ve)^{2m}} \quad \text{for }x\in B_R(0),\; k\geq k(\ve,R).$$
With this information, we can apply the proofs of Lemma \ref{claim3} and Lemma \ref{claim4} to get $(A_2)$ and $(A_3)$ for $i=\ell+1$. Finally, $(A_4)$ follows from $(A_3)$.

\medskip

\noindent\emph{Step 3.} The procedure has to stop, i.e. $(A_5)$ has to be satisfied after a finite number of inductive steps. Indeed at the $\ell$-th steps we get
\begin{eqnarray*}
\lim_{k\to\infty}\int_\Omega e^{2m\hat u_k}dx&\geq&\lim_{R\to\infty}\lim_{k\to\infty}\sum_{i=1}^\ell\int_{B_{R\mu_{k,i}}(x_{k,i})}e^{2m\hat u_k(y)}dy\\
&=&\lim_{R\to\infty}\lim_{k\to\infty} \sum_{i=1}^\ell\int_{B_{R}(0)}e^{2m\tilde u_{k,i}(y)}dy\\
&=&\ell\int_{\R{2m}}e^{2m \eta_0}dx=\ell|S^{2m}|,
\end{eqnarray*}
which, together with \eqref{eq3.2}, gives an upper bound for $\ell$. Setting $N$ to be the $\ell$ at which our inductive procedure stops, we conclude.
\end{proof}

From now on, the $N$ converging sequences
$$\{x_{k,i}\to x^{(i)}:1\leq i\leq N\} $$ produced with Lemma \ref{hp} will be fixed and we shall set
\begin{equation}\label{S}
S:=\{x^{(i)}:1\leq i\leq N\}.
\end{equation}

\begin{lemma}\label{claim7}
For $\ell\in\{1,\ldots,2m-1\}$ there exists $C>0$ such that
\begin{equation}
\inf_{1\leq i\leq \ell}|x-x_{k,i}|^\ell\abs{\nabla^\ell\hat u_k(x)}\leq C,\text{ for }x\in\Omega.
\end{equation}
\end{lemma}

\begin{proof}
As already noticed, we can use \eqref{eq3.3}, \eqref{DAS} and the symmetry of $G$ to get
\begin{equation}\label{int0}
|\nabla^\ell\hat u_k(x)|\leq C\int_\Omega \frac{e^{2m\hat u_k(y)}}{|x-y|^\ell}dy.
\end{equation}
Let $\Omega_{k,i}:=\{x\in\Omega:\;\dist(x,\{x_{k,1},\ldots,x_{k,N}\})=\abs{x-x_{k,i}}\}$, fix $x\in \Omega_{k,i}$, and write
\begin{equation}\label{int}
\int_{\Omega_{k,i}}\frac{e^{2m\hat u_k(y)}}{\abs{x-y}^\ell}dy=\int_{\Omega_{k,i}\cap B_{k,i}}\frac{e^{2m\hat u_k(y)}}{\abs{x-y}^\ell}dy 
+ \int_{\Omega_{k,i}\setminus B_{k,i}}\frac{e^{2m\hat u_k(y)}}{\abs{x-y}^\ell}dy,
\end{equation}
where $B_{k,i}:=B_{\frac{|x-x_{k,i}|}{2}}(x_{k,i})$. By Property $(A_5)$ we get
\begin{eqnarray}
e^{2m\hat u_k(y)}&\leq& C\abs{y-x_{k,i}}^{-2m} \quad \text{for }y\in\Omega_{k,i}\setminus B_{k,i}\label{*1}\\
|x-y|&\geq&\frac{1}{2}\abs{x-x_{k,i}}\quad \text{for }y\in\Omega_{k,i}\cap B_{k,i}.\label{*2}
\end{eqnarray}
Then, using \eqref{eq3.2} and \eqref{*1}, we get
\begin{equation}\label{int2}
\int_{\Omega_{k,i}\cap B_{k,i}}\frac{e^{2m\hat u_k(y)}}{\abs{x-y}^\ell}dy\leq \frac{C}{|x-x_{k,i}|^{\ell}}.
\end{equation}
As for the last integral in \eqref{int}, we write $\Omega_{k,i}\setminus B_{k,i}=\Omega_{k,i}^{(1)}\cup \Omega_{k,i}^{(2)}$, where
$$\Omega_{k,i}^{(1)}=(\Omega_{k,i}\backslash B_{k,i})\cap B_{2|x-x_{k,i}|}(x),\quad \Omega_{k,i}^{(2)}=(\Omega_{k,i}\backslash B_{k,i})\backslash B_{2|x-x_{k,i}|}(x).$$
Then straightforward computations and \eqref{*2} imply
\begin{eqnarray*}
\int_{\Omega_{k,i}\backslash B_{k,i}}\frac{e^{2m\hat u_k(y)}dy}{|x-y|^\ell}&\leq&C\int_{\Omega^{(1)}_{k,i}}\frac{dy}{|y-x_{k,i}|^{2m}|x-y|^\ell}\\
&&+C\int_{\Omega^{(2)}_{k,i}}\frac{dy}{|y-x_{k,i}|^{2m}|x-y|^\ell}\\
&\leq&\frac{C}{|x-x_{k,i}|^{2m}}\int_{\Omega_{k,i}^{(1)}}\frac{dy}{|x-y|^\ell}+C\int_{\Omega_{k,i}^{(2)}}\frac{dy}{|y-x_{k,i}|^{2m+\ell}}\\
&\leq&\frac{C}{|x-x_{k,i}|^\ell}.
\end{eqnarray*}
Summing up with \eqref{int0}, \eqref{int} and \eqref{int2}, the proof is complete. 
\end{proof}

\begin{lemma}\label{claim10} Up to a subsequence, we have
$$\lim_{k\to\infty}\alpha_k=+\infty.$$
\end{lemma}

\begin{proof} We argue by contradiction. Suppose $\lim_{k\to\infty}\alpha_k=\alpha_0\in\R{}$.

\medskip

\noindent\emph{Step 1.} We claim that $S\subset\de\Omega$, where $S$ is as in \eqref{S}, and there is a function $u_0\in C^{2m-1,\alpha}(\overline \Omega)$ such that
$$u_k \to u_0\quad \text{in } C^{2m-1,\alpha}_{\loc}(\overline \Omega\backslash S).$$ Moreover $u_0$ satisfies
\begin{equation}\label{u02}
\left\{\begin{array}{ll}
(-\Delta)^mu_0=(2m-1)!e^{-2m\alpha_0}e^{2mu_0}\text{ in }\Omega\\
u_0=\de_\nu u_0=\ldots=\de^{m-1}_{\nu}u_0=0\text{ in }\de\Omega
\end{array}\right.
\end{equation}
Indeed \eqref{eq3.4} and the assumption that $\alpha_k\to \alpha_0$ imply that
\begin{equation}\label{stimaL1}
\norm{\hat u_k}_{L^1(\Omega)}\leq C.
\end{equation}
Since $\hat u_k$ satisfies \eqref{eqhat} and \eqref{eq3.2}, we can apply Theorem \ref{main} from the appendix. This implies that one of the following is true
\begin{itemize}
\item[(i)] Up to a subsequence, $\hat u_k\to \hat u_0$ in $C^{2m-1,\alpha}_{\loc}(\Omega)$.
\item[(ii)] Up to a subsequence $\hat u_k\to -\infty$ locally uniformly in $\Omega\backslash\Omega_0$ for a set $\Omega_0$ of Hausdorff dimension at most $2m-1$.
\end{itemize}
Clearly case (ii) contradicts \eqref{stimaL1}, hence case (i) occurs and $S\subset\de\Omega$. Using the boundary condition, Lemma \ref{claim7}, and elliptic estimates, we actually infer that $\hat u_k\to \hat u_0$ in $C^{2m-1,\alpha}_{\loc}(\overline \Omega\backslash S)$. Then clearly $u_k\to u_0:=\hat u_0+\alpha_0$ in $C^{2m-1,\alpha}_{\loc}(\overline \Omega\backslash S)$ and $u_0$ satisfies \eqref{u02}.

We finally want to prove that $u_0$ is continuous in $\overline\Omega$, hence smooth.
In the limit as $k\to \infty$, Lemma \ref{claim7} implies
$$\inf_{1\leq i\leq N}|x-x^{(i)}||\nabla u_0(x)|\leq C \quad \text{for } x\in \Omega\setminus S.$$
Fix $x^{(i)}\in S$ and $\delta>0$ such that
$$|x-x^{(i)}|\abs{\nabla u_0(x)}\leq C \quad \text{for } x\in \Omega\cap B_\delta(x^{(i)})\setminus \{x^{(i)}\}.$$
Then there is a constant $C>0$ such that
$$|u(x)-u(y)|\leq C\quad \text{for }x,y\in\Omega\cap B_\delta(x^{(i)})\setminus \{x^{(i)}\},\;|x-x^{(i)}|=|y-x^{(i)}|. $$
By taking $y\in\de\Omega$ and using \eqref{dir}, we obtain that $u$ is bounded near $x^{(i)}$. Then \eqref{u02} and elliptic regularity imply that $u_0\in C^\infty(\overline\Omega)$.

\medskip

\noindent\emph{Step 2.} If $S=\emptyset$, then Step 1 yields $u_k\to u_0$ in $C^{2m-1,\alpha}_{\loc}(\overline \Omega)$, which contradicts the assumption $\sup_\Omega u_k\to+\infty$. If instead there exists $x_0\in S\subset\de\Omega$ and take $\delta>0$ such that $S\cap B_\delta(x_0)=\{x_0\}$. 
Set for $0< r\leq\delta$
\begin{equation}\label{rhokr}
 \rho_{k,r}=\frac
{\int_{\de\Omega\cap B_r(x_0)}
(x-x_0)\cdot\nu(x)|\Delta^\frac{m}{2} u_k|^2}
{\int_{\de\Omega\cap B_r(x_0)}
\nu(x_0)\cdot\nu(x)|\Delta^\frac{m}{2}u_k|^2},
\end{equation}
where for $m$ odd we put $\Delta^\frac{m}{2} u_k:=\nabla \Delta^\frac{m-1}{2}u_k\in\R{2m}$ (compare \eqref{nablam} below), $\nu(x)$ denotes the exterior normal to $\de\Omega$ at $x$, and we assume that the denominator in \eqref{rhokr} does not vanish, otherwise we simply set $\rho_{k,r}=r$. Set also
\begin{equation}\label{ykr}
y_{k,r}:= x_0+\rho_{k,r}\nu(x_0).
\end{equation}
Up to taking $\delta$ even smaller, we may assume that
\begin{equation}\label{1/2}
\frac{1}{2}\leq\nu(x_0)\cdot\nu(x)\leq 1 \quad \text{for } x\in\de \Omega\cap \overline B_r(x_0),\;r\leq\delta,
\end{equation}
hence $|\rho_{k,r}|\leq 2r.$
Applying Lemma \ref{poho} to $u_k$ on the domain $\Omega':=\Omega\cap B_r(x_0)$, with
$$Q=(2m-1)!e^{-2m\alpha_k},\quad y=y_{k,r},$$
and by the property $(A_4)$, we get
\begin{eqnarray}
\Lambda_1&\leq&\lim_{k\to\infty}(2m-1)!\int_{\Omega'}e^{2m\hat u_k}dx\nonumber\\
&=&\lim_{k\to\infty}\frac{(2m-1)!}{2m}\int_{\partial \Omega'}(x-y_{k,r})\cdot\nu_{\Omega'} e^{2m \hat u_k}d\sigma\label{s3m}\\
&&-\lim_{k\to\infty}\frac{1}{2}\int_{\partial \Omega'}(x-y_{k,r})\cdot\nu_{\Omega'} |\Delta^\frac{m}{2} u_k|^2d\sigma  +\lim_{k\to\infty}\int_{\partial \Omega'}f_kd\sigma,\nonumber
\end{eqnarray}
where $f_k$ is defined on $\partial\Omega'$ by
\begin{equation}\label{eqfk}
f_k(x)=\sum_{j=0}^{m-1}(-1)^{m+j+1}\nu_{\Omega'}\cdot\Big(\Delta^\frac{j}{2}((x-y_{k,r})\cdot \nabla u_k(x))\Delta^\frac{2m-1-j}{2}u_k(x)\Big).
\end{equation}
Now write $f_k=f_k^{(1)}+f_k^{(2)}$, where
\begin{equation}\label{f22k}
 f_k^{(2)}(x)=\left\{\begin{array}{ll}
                      \nu_{\Omega'}\cdot\Delta^\frac{m}{2}u_k(x)\, (x-y_{k,r})\cdot \Delta^\frac{m}{2}u_k(x)&\text{ if }m\text{ is odd}\\
                      D^2\Delta^\frac{m-2}{2}u_k(x)(\nu_{\Omega'},x-y_{k,r})\,\Delta^\frac{m}{2}u_k(x)&\text{ if }m\text{ is even,}
                     \end{array}\right.
\end{equation}
where we use the notation $D^2\varphi(x)(\xi,\zeta):=\frac{\de^2 \varphi(x)}{\de x^i\de x^j}\xi^i\zeta^j.$
Using \eqref{form1} below, one can see that
\begin{eqnarray*}
f_k^{(1)}(x)&=&\sum_{j=0}^{m-2}(-1)^{m+j+1}\nu_{\Omega'}\cdot\Big(\Delta^\frac{j}{2}((x-y_{k,r})\cdot \nabla u_k(x))\Delta^\frac{2m-1-j}{2}u_k(x)\Big)\\
&&+\;g_k(x),\\
g_k(x)&=&\left\{\begin{array}{ll}
                 (m-1)\nu_{\Omega'}(x)\cdot \Delta^\frac{m}{2}u_k(x)\,\Delta^\frac{m-1}{2}u_k(x)&\text{ if }m\text{ is odd}\\
                 (m-1)\nu_{\Omega'}(x)\cdot \Delta^\frac{m-1}{2}u_k(x)\,\Delta^\frac{m}{2}u_k(x)&\text{ if }m\text{ is even.}
                \end{array}\right.
\end{eqnarray*}
Notice that \eqref{dir} implies that $\nabla^\ell u_k=0$ on $\de\Omega$ for $0\leq \ell\leq m-1$.
Since each monomial of $f_k^{(1)}$ contains a factor of the form $\partial^\gamma u_k$ for some multi-index $\gamma$ with $|\gamma|\leq m-1$,
we get
$$\int_{\partial\Omega\cap B_r(x_0)}f_k^{(1)}d\sigma=0.$$
We now claim that
\begin{multline}\label{fk22}
\int_{\partial \Omega\cap B_r(x_0)}\bigg[-\frac{1}{2}(x-y_{k,r})\cdot\nu_{\Omega} |\Delta^\frac{m}{2} u_k|^2+f_k^{(2)}\bigg]d\sigma\\
=\frac{1}{2}\int_{\partial \Omega\cap B_r(x_0)}(x-y_{k,r})\cdot\nu_{\Omega} |\Delta^\frac{m}{2} u_k|^2d\sigma.
\end{multline}
It $m$ is odd, $\Delta^\frac{m-1}{2}u_k\equiv 0$ on $\de \Omega$ implies that $\Delta^\frac{m}{2} u_k(x)\perp \de\Omega$ for $x\in\de \Omega$, whence
$$f_k^{(2)}(x)=\nu_{\Omega'}\cdot \Delta^\frac{m}{2}u_k\,(x-y_{k,r})\cdot \Delta^\frac{m}{2}u_k=\nu_{\Omega'}\cdot (x-y_{k,r})|\Delta^\frac{m}{2}u_k|^2, \quad x\in\de\Omega.$$
Then \eqref{fk22} follows.
When $m$ is even, we also have
\begin{equation}\label{nuk}
f_k^{(2)}(x)=\nu_{\Omega'}\cdot (x-y_{k,r})|\Delta^\frac{m}{2}u_k|^2\quad \text{on }\de \Omega.
\end{equation}
To see that, write $U_k:=\Delta^\frac{m-2}{2}u_k$. Then $U_k\equiv 0$ and $\nabla U_k\equiv0$ on $\de \Omega$, hence
$$D^2 U_k(x)=\nu^i_{\Omega}\nu^j_{\Omega} \Delta U_k,$$
\eqref{nuk} is proven and \eqref{fk22} follows.

\medskip

Now, the second integral in \eqref{fk22} must be zero by \eqref{rhokr} and \eqref{ykr}, if the denominator in \eqref{rhokr} does not vanish. If it vanishes, observe that, by \eqref{1/2}
$$
\nu(x_0)\cdot \nu(x) |\Delta^\frac{m}{2}u_k|^2\geq \frac{1}{2}|\Delta^\frac{m}{2}u_k|^2,
$$
therefore we obtain $\Delta^\frac{m}{2}u_k=0$ on $\de\Omega\cap B_r(x_0)$, and also in this case the integrals in \eqref{fk22} vanish.

\medskip

By \eqref{dir} and Lemma \ref{claim1}, we also have
$$\bigg|\frac{(2m-1)!}{2m}\int_{\partial \Omega\cap B_r(x_0)}(x-y_{k,r})\cdot \nu_{\Omega'}e^{2m\hat u_k}  \bigg|\leq C\int_{\partial \Omega\cap B_r(x_0)} r e^{-2m\alpha_k}\leq Cr^{2m}.$$
All the other terms on the right-hand side of \eqref{s3m}, namely the integrals over $\Omega\cap\partial B_r(x_0)$, are bounded by $C r^{2m-1}$ for $0<r\leq\delta$ and $k\geq k(r)$ large enough. Indeed, by Step 1 we have
$$\lim_{k\to\infty} \sup_{\partial B_r(x_0)\cap \Omega }|\nabla^\ell u_k-\nabla^\ell u_0|=0, \quad |\nabla^\ell u_0|\leq C,\quad 0\leq \ell\leq 2m-1.$$
Therefore, taking the limit as $k\to 0$ first and $r\to 0$ then, we infer
$$\Lambda_1\leq Cr^{2m-1}.$$
This gives a contradiction as $r\to 0$, hence completing the proof in the case when $m$ is odd.

\end{proof}

\begin{lemma}\label{claim9}
Up to selecting a subsequence,
\begin{equation}\label{menoinf}
\hat u_k\rightarrow-\infty \quad\text{locally uniformly on }\overline \Omega\setminus S,
\end{equation}
where $S$ is as in \eqref{S}. Moreover
\begin{equation}\label{limuk}
 \lim_{k\rightarrow +\infty}u_k=\sum_{i=1}^N\beta_i G_{x^{(i)}}\text{ in }C^{2m-1,\alpha}_{\loc}(\bar\Omega\setminus S),
\end{equation}
with
\begin{equation}\label{beta}
\beta_i:=(2m-1)!\lim_{\delta\rightarrow 0}\lim_{k\rightarrow\infty}\int_{B_{\delta}(x^{(i)})\cap\Omega}e^{2m\hat u_k}dy,
\end{equation}
and $\beta_i \geq \Lambda_1$, for $1\leq i\leq N$.
\end{lemma}

\begin{proof} \emph{Step 1.} We claim that $\hat u_k\to-\infty$ locally uniformly on $\overline\Omega\backslash S$. Indeed take $\delta>0$ such that $\Omega_\delta:=\Omega\setminus\cup_{i=1}^N\overline B_\delta(x_i)$ is connected and $\de \Omega_\delta\cap\de\Omega\neq\emptyset$. Lemma \ref{claim7} implies that $\hat u_k$ is Lipschitz on $\Omega_\delta$, and we also have $\hat u_k=-\alpha_k$ on $\de\Omega_\delta\cap\de\Omega$, hence
\begin{equation}\label{ukbdd}
|u_k|=|\hat u_k +\alpha_k|\leq C_\delta \textrm{ in }\overline\Omega_\delta.
\end{equation}
Since $\alpha_k\to+\infty$, we have $\hat u_k\to-\infty$ uniformly on $\overline \Omega_\delta$, hence the claim is proved.

\medskip

\noindent\emph{Step 2.}
By \eqref{dir} and Lemma \ref{claim7}, the $u_k$'s are bounded in $C^0_{\loc}(\overline\Omega\setminus S)$. Since
$$(-\Delta)^m u_k=(2m-1)!e^{-2m\alpha_k}e^{2mu_k},$$
where the right-hand side is bounded $C^0_{\loc}(\overline\Omega\backslash S)$,
by elliptic regularity we have that, up to a subsequence,
$$u_k\to \psi \quad \text{in } C^{2m-1,\alpha}_{\loc}(\overline\Omega\backslash S),$$ for some $\psi\in C^{2m-1,\alpha}_{\loc}(\overline\Omega\setminus S)$. Up to taking $\delta>0$ smaller, we may assume that $\overline{B_{\delta}(x^{(i)})}\cap \overline{B_{\delta}(x^{(j)})}=\emptyset$ for $i\neq j$. Since $\hat u_k\rightarrow -\infty$ uniformly on the compact $\overline\Omega_{\delta}$, we have by \eqref{eq3.3} 
\begin{eqnarray}
\lim_{k\rightarrow\infty}u_k(x)&=& (2m-1)!\lim_{k\rightarrow\infty}\int_\Omega G_x(y)e^{2m\hat u_k(y)}dy\nonumber\\
&=& (2m-1)!\lim_{k\rightarrow\infty}\sum_{i=1}^N\int_{B_{\delta}(x^{(i)})\cap\Omega}G_x(y)e^{2m\hat u_k(y)} dy.\label{green2}
\end{eqnarray}
Now we want an explicit expression for $\psi$. Fix $x\in \overline \Omega\backslash S$. We observe that $G(x,\cdot)$ is smooth away from $x$; in particular it is continuous on $B_{\delta}(x^{(i)})$ for all $i$ (up to decreasing $\delta$). By \eqref{eq3.2}, up to a subsequence we have
$$e^{2m\hat u_k}(y)dy\rightharpoonup \nu\quad \text{in }\overline\Omega$$
weakly in the sense of measures, for some positive Radon measure $\nu$.
On the other hand, since \eqref{menoinf} implies that the support of $\nu$ is contained in $S$, we get
$$\nu=\sum_{i=i}^N\beta_i\delta_{x^{(i)}},$$
for some constants $\beta_i\geq 0$.
Then \eqref{green2} implies
$$
\lim_{k\rightarrow\infty} u_k(x)=\sum_{i=1}^N\beta_i G_{x^{(i)}}(x)\quad\forall x\in\Omega\setminus S,
$$
where $\beta_i$ is as in \eqref{beta}.
Now we fix a point $x^{(i)}\in S$ and we set $\mu_{k,i}$ and $x_{k,i}$ as in Lemma \ref{claim4}. By ($A_4$)
$$\lim_{k\rightarrow \infty}\int_{B_{\delta}(x^{(i)})\cap\Omega}e^{2m\hat u_k(x)}dx\geq\lim_{R\to\infty}\lim_{k\rightarrow \infty}\int_{B_{R\mu_k}(x_{k,i})}e^{2m\hat u_k(x)}dx=|S^{2m}|.
$$
Taking the limit as $\delta\to 0$ we get $\beta_i\geq \Lambda_1$, as claimed.
\end{proof}

\begin{lemma}\label{claim11}
For any $x_0\in\de\Omega$ we have
\begin{equation}\label{r0}
 \lim_{r\rightarrow 0}\lim_{k\rightarrow+\infty}\int_{B_r(x_0)\cap\Omega}e^{2m\hat u_k}dx=0.
\end{equation}
In particular $S\cap\de\Omega=\emptyset$.
\end{lemma}

\begin{proof}
Fix $x_0\in\de\Omega$. If $x_0\not\in S$, then \eqref{r0} follows at once from Lemma \ref{claim9}. Then we can assume $x_0=x^{(j)}\in \de\Omega\cap S$ for some $1\leq j\leq N$, and proceed by contradiction. Take $\delta>0$ such that $S\cap B_{\delta}(x_0)=\{x_0\}$. Let $\nu:\de\Omega\rightarrow S^{2m-1}$ be the outward pointing normal to $\de\Omega$. Set $\rho_{k,r}$ and $y_{k,r}$ as in \eqref{rhokr} and \eqref{ykr}.
Take $r>0$ so small that
$$\frac{1}{2}\leq\nu(x_0)\cdot\nu(x)\leq 1 \quad \text{for } x\in\de\Omega\cap\overline B_r(x_0),$$
so that $|\rho_{k,r}|\leq 2r$.
Applying Lemma \ref{poho} to $u_k$ on the domain $\Omega':=\Omega\cap B_r(x_0)$, with
$$Q=(2m-1)!e^{-2m\alpha_k},\quad y=y_{k,r},$$
we obtain
\begin{eqnarray*}\label{int12}
(2m-1)!\int_{\Omega'}e^{2m\hat u_k}dx&=&\frac{(2m-1)!}{2m}\int_{\partial \Omega'}(x-y_{k,r})\cdot\nu_{\Omega'} e^{2m \hat u_k}d\sigma\\
&&+\int_{\partial \Omega'}\bigg[-\frac{1}{2}(x-y_{k,r})\cdot\nu_{\Omega'} |\Delta^\frac{m}{2} u_k|^2+f_k^{(2)}(x)\bigg]d\sigma  \\
&&+\int_{\partial \Omega'}f_k^{(1)}(x)d\sigma,\nonumber
\end{eqnarray*}
where
$f_k(x)=f_k^{(1)}+f_k^{(2)}$, with the same notations as in \eqref{eqfk}, \eqref{f22k}.
Since each monomial of $f_k^{(1)}$ contains a factor of the form $\partial^\gamma u_k$ with $|\gamma|\leq m-1$,
we get
$$\int_{\partial\Omega\cap B_r(x_0)}f_k^{(1)}d\sigma=0.$$
Again we have that \eqref{fk22} holds and the corresponding integral vanishes, thanks to our choice of $\rho_{k,r}$ and $y_{k,r}$. This takes care of the integral on $\de\Omega\cap B_r(x_0)$.

Since $G_{x_0}\equiv 0$, and the derivatives of $G_{x^{(i)}}$ are bounded in $\overline{B_r(x_0)}$ for $x^{(i)}\neq x_0$, \eqref{limuk} implies
$$
\lim_{k\rightarrow+\infty}\int_{\Omega\cap\de B_r(x_0)}f_k^{(1)} d\sigma\leq Cr^{2m-1},
$$
and
$$\lim_{k\to\infty}\int_{\Omega\cap \partial B_r(x_0)}\bigg[-\frac{1}{2}(x-y_{k,r})\cdot\nu|\Delta^\frac{m}{2} u_k|^2+f_k^{(2)}\bigg]d\sigma\leq Cr^{2m}.$$
As for the first term on the right-hand side of \eqref{int12}, \eqref{dir} and Lemma \ref{dir} imply
$$
\int_{\de\Omega'}(x-y_{k,r})\cdot\nu_{\Omega'}e^{-2m\alpha_k}e^{2mu_k}d\sigma\leq Cr^{2m}.
$$
Summing up all the contributions and letting $r\rightarrow 0$ we get \eqref{r0}.
\end{proof}

\begin{lemma}\label{claim13}
In \eqref{limuk} and \eqref{beta} we have $\beta_i=\Lambda_1$ for all $1\leq i\leq N$.
\end{lemma}

\begin{proof}
Since $S\cap\de\Omega=\emptyset$, there exists $\delta>0$ such that $B_{\delta}(x^{(i)})\subset\Omega$, and $S\cap B_{\delta}(x^{(i)})=\{x^{(i)}\}$ for all $1\leq i\leq N$. Fix $i$ with $1\leq i\leq N$ and suppose, up to a translation, that $x^{(i)}=0$. Recall that
$$
\beta_i=(2m-1)!\lim_{\delta\rightarrow 0}\lim_{k\rightarrow \infty}\int_{B_{\delta}(0)}e^{2m\hat u_k}dx.
$$ 
By the Pohozaev identity of Lemma \ref{poho}, applied to $u_k$ on the domain $B_\delta:=B_{\delta}(0)$ with $y=0$ and $Q=(2m-1)!e^{-2m\alpha_k}$, we get
\begin{equation}\label{eqpoh}
(2m-1)!\int_{B_{\delta}}e^{2m\hat u_k}dx=I_\delta(u_k)+II_\delta(u_k)+III_\delta(u_k),
\end{equation}
where
\begin{eqnarray*}
 I_\delta(u_k)&=&\frac{\delta(2m-1)!}{2m}\int_{\de B_{\delta}} e^{2m\hat u_k}d\sigma\\
II_\delta(u_k)&=&-\frac{\delta}{2}\int_{\de B_{\delta}}|\Delta^\frac{m}{2} u_k|^2 d\sigma\\
III_\delta(u_k)&=&\sum_{j=0}^{m-1}(-1)^{m+j+1}\int_{\de B_\delta}\nu\cdot\left(\Delta^\frac{j}{2}\left(x\cdot \nabla u_k\right)\Delta^\frac{2m-1-j}{2}u_k \right)d\sigma
\end{eqnarray*}
From Lemma \ref{claim9} we infer
\begin{eqnarray*}
\lim_{k\rightarrow\infty}II_\delta(u_k)&=&II_\delta(\beta_i G_0)=\beta_i^2II_\delta(G_0)\\
\lim_{k\rightarrow\infty}III_\delta(u_k)&=&III_\delta(\beta_i G_0)=\beta_i^2III_\delta(G_0).
\end{eqnarray*}
Since the functions $e^{2m\hat u_k}\to 0$ in $C^0(\de B_\delta)$, we have 
$$
\lim_{k\rightarrow\infty}I_\delta(u_k)=0.
$$
The Green function $G_0$ can be decomposed in the sum of a fundamental solution for the operator $(-\Delta)^m$ on $\R{2m}$ and a so-called regular part $R$, which is smooth:
Let us write
$$G_0=g+R\quad \text{in }\overline\Omega$$
where
$$g(x):=\frac{1}{\gamma_{2m}}\log\frac{1}{|x|},\quad \gamma_{2m}:=\frac{\Lambda_1}{2}$$
satisfies $(-\Delta)^m g=\delta_0$ (see e.g. Proposition 22 in \cite{mar1}),
and $R:=G_0-g\in C^\infty(\overline\Omega)$.
Since
\begin{equation}\label{R}
|\nabla^j R|\leq C,\quad |\nabla^j g|\leq \frac{C}{\delta^j}\quad \text{on }\de B_\delta,
\end{equation}
we get
$$ II_\delta(R+g)-II_\delta(g)\leq C \delta \int_{\de B_{\delta}}C\left(|\Delta^\frac{m}{2} g| +C\right)d\sigma\leq C \delta^m.$$
For the terms in $III_\delta(R+g)$, \eqref{R} implies
\begin{eqnarray*}
III_\delta^{(j)}(g+R)&:=& \int_{\de B_\delta}\nu\cdot\left(\Delta^\frac{j}{2}\left(x\cdot \nabla (R+g)\right)\Delta^\frac{2m-1-j}{2}(R+g)\right)d\sigma\\
&=&\int_{\de B_\delta}\nu\cdot\left(\Delta^\frac{j}{2}\left(x\cdot \nabla g\right)\Delta^\frac{2m-1-j}{2}g\right)d\sigma\\
&&+\int_{\de B_\delta}\nu\cdot\left(\Delta^\frac{j}{2}\left(x\cdot \nabla R\right)\Delta^\frac{2m-1-j}{2}g\right)d\sigma\\
&&+\int_{\de B_\delta}\nu\cdot\left(\Delta^\frac{j}{2}\left(x\cdot \nabla g\right)\Delta^\frac{2m-1-j}{2}R\right)d\sigma\\
&&+\int_{\de B_\delta}\nu\cdot\left(\Delta^\frac{j}{2}\left(x\cdot \nabla R\right)\Delta^\frac{2m-1-j}{2}R\right)d\sigma\\
&=&III_\delta^{(j)}(g)+O(\delta)\quad \text{as }\delta\to 0,
\end{eqnarray*}
where $|O(\delta)|\leq C\delta$ as $\delta\to 0$.
Summing up all what we proved until now, we obtain

$$\beta_i=\beta_i^2\lim_{\delta\rightarrow 0}\lim_{k\to\infty}\big[I_\delta(u_k)+II_\delta(u_k)+III_\delta(u_k)\big]=\beta_i^2 \lim_{\delta\to 0}\big[II_\delta(g)+III_\delta(g)\big].$$
On the other hand, since $II_\delta(g)$ and $III_\delta(g)$ do not depend on $\delta$, it is enough to compute

\begin{equation}\label{eqg}
\beta_i=II_\delta(g)+III_\delta(g)
\end{equation}
for an arbitrary $\delta>0$.
Using the formula
$$\gamma_{2m}\Delta^k g=(-1)^k(2k-2)!!\frac{(2m-2)!!}{(2m-2k-2)!!}r^{-2k},$$
we find
$$
II_\delta(g)=-\frac{\delta}{2}\int_{\de B_\delta}\left[\frac{(2m-2)!!}{\gamma_{2m}}r^{-m}\right]^2d\sigma=-|S^{2m-1}|\frac{[(2m-2)!!]^2}{2\gamma_{2m}^2}.
$$
Observing that
\begin{eqnarray*}
\Delta^k(x\cdot \nabla g)&=&2k\Delta^kg+r\de_r\Delta^k g=0,\\
\de_r(x\cdot\nabla g)&=&-r^{-1} - x\cdot\nabla(r^{-1})=0,\\
x\cdot\nabla g&=&r\de_rg=-\frac{1}{\gamma_{2m}},\\
\gamma_{2m}\de_r\Delta^k g&=&(-1)^{k+1}(2k)!!\frac{(2m-2)!!}{(2m-2k-2)!!}r^{-2k-1} \\
\end{eqnarray*}
we see that $III_\delta^{(j)}(g)=0$ for $1\leq j\leq m-1$, and 
\begin{eqnarray*}
III_\delta(g)&=&III_\delta^{(0)}(g)=(-1)^{m+1}\int_{\de B_\delta}(x\cdot \nabla g)\de_r\Delta^{m-1}gd\sigma\\
&=&|S^{2m-1}|\frac{[(2m-2)!!]^2}{\gamma_{2m}^2}.
\end{eqnarray*}
From \eqref{eqg} we get
$$
\frac{1}{\beta_i}=|S^{2m-1}|\frac{[(2m-2)!!]^2}{2\gamma_{2m}^2}=\frac{1}{(2m-1)!|S^{2m}|},
$$
whence $\beta_i=\Lambda_1$.
\end{proof}

\medskip

\noindent\emph{Proof of Theorem \ref{trm2}.} By Corollary \ref{coru}, it suffices to prove that, under the assumption \eqref{ukinfty}, case (ii) of the theorem occurs. This follows at once putting together Lemmas \ref{hp}, \ref{claim9}, \ref{claim11} and \ref{claim13}.
\phantom{ }\hfill$\square$

\appendix
\section*{Appendix}

\subsection*{A useful theorem}

Several times we used the following theorem from \cite{mar2} (compare also \cite{BM} and \cite{ARS}).

\begin{trm}\label{main} Let $\Omega$ be a domain in $\R{2m}$, $m>1$, and let $(u_k)_{k\in \mathbb{N}}$ be a sequence of functions satisfying
\begin{equation}\label{eqliou}
(-\Delta)^m u_k=(2m-1)! e^{2m u_k}.
\end{equation}
Assume that
\begin{equation}\label{vk}
\int_{\Omega} e^{2m u_k}dx\leq C,
\end{equation}
for all $k$ and define the finite (possibly empty) set 
$$S_1:=\bigg\{x\in\Omega : \lim_{r\to 0^+}\lim_{k\to\infty}\int_{B_r(x)}(2m-1)!e^{2mu_k}dy\geq \frac{\Lambda_1}{2} \bigg\}.$$
Then one of the following is true.
\begin{enumerate}
\item[(i)] A subsequence converges in $C^{2m-1,\alpha}_{\loc}(\Omega)$ and $S_1=\emptyset$.

\item[(ii)] There exist a subsequence, still denoted by $(u_k)$, a closed nowhere dense set $S_0$ of Hausdorff dimension at most $2m-1$ such that, letting $\Omega_0=S_0\cup S_1,$
we have $u_k\rightarrow -\infty$ locally uniformly in $\Omega\bs \Omega_0$ as $k\rightarrow \infty.$ Moreover there is a sequence of numbers $\beta_k\rightarrow \infty$ such that
$$\frac{u_k}{\beta_k}\rightarrow \varphi \textrm{ in } C^{2m-1,\alpha}_{\loc}(\Omega\backslash \Omega_0),$$
where $\varphi\in C^\infty(\Omega\bs S_1)$, $S_0=\{x\in\Omega:\varphi(x)=0\}$, and
$$(-\Delta)^m\varphi\equiv 0,\quad \varphi\leq 0,\quad \varphi \not\equiv 0\quad \textrm{in }\Omega\backslash S_1.$$
\end{enumerate}
\end{trm}

\subsection*{Pohozaev's identity}

We now discuss a generalization of the celebrated Pohozaev identity to higher dimension, Lemma \ref{poho} below. A similar identity can be also found in \cite{xu}. 
We use the following notation:

\begin{equation}\label{nablam}
\Delta^\frac{m}{2} u:= \nabla\Delta^n u   \in\R{2m}   \text{ if }m=2n+1\text{ is odd,}
\end{equation}
and we define $\Delta^j u\cdot\Delta^\ell u$ using the inner product of $\R{2m}$, or the multiplication by a scalar, or the product of $\R{}$, according to whether $j$ and $\ell$ are integer or half-integer.

Preliminary to the proof of Pohozaev's identity, we need the following lemma.

\begin{lemma}\label{alg}
Let $u\in C^{m+1}(\Omega)$, where $\Omega\subset \R{2m}$ is open, and let $y\in \R{2m}$ be fixed. We have
$$
\frac{1}{2}\diver((x-y)|\Delta^{\frac{m}{2}} u|^2)= \Delta^{\frac{m}{2}}((x-y)\cdot \nabla u) \cdot \Delta^{\frac{m}{2}} u
$$
\end{lemma}

\begin{proof} By a simple translation we can assume $y=0$. Let us first assume $m$ even. Then
\begin{eqnarray}
\frac{1}{2}\diver(x |\Delta^\frac{m}{2} u|^2)&=&m|\Delta^\frac{m}{2} u|^2+ \big[(x\cdot \nabla) \Delta^\frac{m}{2} u)\big]\cdot \Delta^\frac{m}{2} u \nonumber\\
&=&m(\Delta^\frac{m}{2}u+(x\cdot\nabla )\Delta^\frac{m}{2} u)\cdot \Delta^\frac{m}{2}u.\label{form0}
\end{eqnarray}
Observing that $D^2x=0$ and use the Leibniz's rule, we also get
\begin{eqnarray}
(x\cdot \nabla) \Delta^\frac{m}{2} u+m \Delta^{\frac{m}{2}}u&=&
(x\cdot \nabla) \Delta^\frac{m}{2} u+m\sum_{i,j=1}^{2m}\partial_{x^j}x^i\Delta^{\frac{m}{2}-1}\partial_{x_j}\partial_{x_i}u\nonumber\\
&=&\Delta^\frac{m}{2}(x\cdot \nabla u)\label{form1}
\end{eqnarray}
Inserting \eqref{form1} into \eqref{form0} we conclude.
\end{proof}

\begin{lemma}\label{poho} Let $u\in C^{m+1}(\overline{\Omega})$, $Q\in \R{}$ satisfy
$$(-\Delta)^m u=Qe^{2mu}$$ in $\Omega\subset \R{2m}$. Let $y\in\R{2m}$ be fixed. Then
\begin{eqnarray*}
\int_{\Omega}Qe^{2mu}dx&=&\frac{1}{2m}\int_{\partial \Omega}(x-y)\cdot\nu Qe^{2m u}d\sigma
-\frac{1}{2}\int_{\partial \Omega}(x-y)\cdot\nu |\Delta^{\frac{m}{2}} u|^2d\sigma\\
&& +\sum_{j=0}^{m-1}(-1)^{m+j+1}\int_{\partial \Omega}\nu\cdot \Big(\Delta^\frac{j}{2}((x-y)\cdot \nabla u) \Delta^{\frac{2m-1-j}{2}}u\Big) d\sigma.\\
\end{eqnarray*}
\end{lemma}

\begin{proof} The proof is a pretty straightforward application of integration by parts. We have

$$\int_{\partial \Omega}(x-y)\cdot \nu Qe^{2mu}d\sigma=\int_{\Omega} 2me^{2mu}Qdx+\int_{\Omega}2m((x-y)\cdot\nabla u)e^{2mu }Qdx,$$
since both sides are equal to $\int_{\Omega}\diver((x-y)e^{2mu})Qdx$. Then we use

\begin{eqnarray*}
\int_{\Omega}(x-y)\cdot\nabla ue^{2mu }Qdx&=&(-1)^m \int_{\Omega}(x-y)\cdot \nabla u\Delta^m u dx\\
&=&\int_{\Omega}\underbrace{\Delta^\frac{m}{2}((x-y)\cdot \nabla u)\Delta^{m}{2} u}_{=\frac{1}{2}\diver ((x-y)|\Delta^\frac{m}{2} u|^2)}dx +\int_{\de \Omega} fd\sigma,\\
\end{eqnarray*}
where
$$f(x):=\sum_{j=0}^{m-1}(-1)^{m+j} \nu\cdot\left(\Delta^\frac{j}{2}\left((x-y)\cdot \nabla u(x)\right)\Delta^\frac{2m-1-j}{2}u (x)\right),\quad x\in\partial\Omega.$$
Moreover
$$
\frac{1}{2}\int_{\Omega}\diver((x-y)|\Delta^\frac{m}{2} u|^2)dx=\frac{1}{2}\int_{\partial \Omega} (x-y)\cdot \nu|\Delta^\frac{m}{2} u|^2d\sigma.
$$
Summing together we conclude.
\end{proof}



\begin{thebibliography}{2}


\bibitem[ARS]{ARS} \textsc{Adimurthi, F. Robert, M. Struwe} \emph{Concentration phenomena for Liouville's equation in dimension 4}, J. Eur. Math. Soc. \textbf{8} (2006), 171-180.
\bibitem[ADN]{ADN} \textsc{S. Agmon, A. Douglis, L. Niremberg} \emph{Estimates near the boundary for solutions of elliptic partial differential equations satisfying general boundary conditions}, Comm. Pure Appl. Math. \textbf{12} (1959), 623-727.
\bibitem[ACL]{ACL} \textsc{N. Aronszaja, T. Creese, L. Lipkin} \emph{Polyharmonic functions}, Clarendon Press, Oxford, 1983.
\bibitem[BM]{BM} \textsc{H. Br\'ezis, F. Merle} \emph{Uniform estimates and blow-up behaviour for solutions of $-\Delta u=V(x)e^u$ in two dimensions}, Comm. Partial Differential Equations \textbf{16} (1991), 1223-1253.
\bibitem[Cha]{cha} \textsc{S-Y. A. Chang} \emph{Non-linear Elliptic Equations in Conformal Geometry}, Zurich lecture notes in advanced mathematics, EMS (2004).
\bibitem[CC]{CC} \textsc{S-Y. A. Chang, W. Chen} \emph{A note on a class of higher order conformally covariant equations}, Discrete Contin. Dynam. Systems \textbf{63} (2001), 275-281.
\bibitem[CL]{CL} \textsc{W. Chen, C. Li} \emph{Classification of solutions of some nonlinear elliptic equations}, Duke Math. J. \textbf{63} (3) (1991), 615-622.
\bibitem[DAS]{DAS} \textsc{A. Dall'Acqua, G. Sweers} \emph{Estimates for Green function and Poisson kernels of higher-order Dirichlet boundary value problems}, J. Differential Equations \textbf{205} (2004), 466-487.
\bibitem[GM]{GM} \textsc{M. Giaquinta, L. Martinazzi} \emph{An introduction to the regularity theory for elliptic systems, harmonic maps and minimal graphs}, Edizioni della Normale, Pisa (2005).
\bibitem[GT]{GT} \textsc{D. Gilbarg, N. Trudinger} \emph{Elliptic partial differential equations of second order}, Springer (1977).
\bibitem[Mar1]{mar1} \textsc{L. Martinazzi} \emph{Classifications of solutions to the higher order Liouville's equation in $\R{2m}$}, Math. Z.
\bibitem[Mar2]{mar2} \textsc{L. Martinazzi} \emph{Concentration-compactness phenomena in higher order Liouville's equation}, preprint (2008).
\bibitem[NS]{NS} \textsc{K. Nagasaki, T. Suzuki} \emph{Asymptotic analysis for two-dimensional elliptic eigenvalue problems with exponentially dominated nonlinearity}, Asymptotic Analysis. \textbf{3} (1990), 173-188.
\bibitem[Rob]{Rob} \textsc{F. Robert} \emph{Quantization effects for a fourth order equation of exponential growth in dimension four}, Proc. Roy. Soc. Edinburgh Sec. A \textbf{137} (2007), 531-553.
\bibitem[RS]{RS} \textsc{F. Robert, M. Struwe} \emph{Asymptotic profile for a fourth order PDE with critical exponential growth in dimension four}, Adv. Nonlin. Stud. \textbf{4} (2004), 397-415.
\bibitem[RW]{RW} \textsc{F. Robert, J.-C. Wei} \emph{Asymptotic behavior of a fourth order mean field equation with Dirichlet boundary condition} (2007), to appeear in Indiana Univ. Math. J.
\bibitem[Xu]{xu} \textsc{X. Xu} \emph{Uniqueness and non-existence theorems for conformally invariant equations}, J. Funct. Anal. \textbf{222} (2005), 1-28.
\bibitem[Wei]{wei} \textsc{J.-C. Wei} \emph{Asymptotic behavior of a nonlinear fourth order eigenvalue problem}, Comm. Partial Differential Equations \textbf{21} (1996), 1451-1467.
\end{thebibliography}
\end{document}